\documentclass{amsart}
\usepackage{amssymb,amsmath}

\usepackage[english]{babel}

\newtheorem{theorem}{Theorem}[section]

\newtheorem{proposition}[theorem]{Proposition}
\newtheorem{corollary}[theorem]{Corollary}

\theoremstyle{definition}

\theoremstyle{remark}

\numberwithin{equation}{section}

\begin{document}

\title{On Affine Interest Rate Models} 

\author{{\bf Paul Lescot\rm}}

\address{Laboratoire de Math\'ematiques Rapha\"el Salem \\
UMR 6085 CNRS \\ 
Universit\'e de Rouen \\
Technop\^ole du Madrillet \\
Avenue de l'Universit\'e, B.P. 12 \\
76801 Saint-Etienne-du-Rouvray (FRANCE) \\ 
T\'el. 00 33 (0)2 32 95 52 24 \\
Fax 00 33 (0)2 32 95 52 86 \\
Paul.Lescot@univ-rouen.fr \\
}

\date{November 17th, 2010}

\setcounter{section}{0}
\begin{abstract}

Bernstein processes are Brownian diffusions that appear in Euclidean Quantum Mechanics.
The consideration of the symmetries of the associated Hamilton-Jacobi-Bellman equation 
allows one to obtain various relations between stochastic processes (Lescot-Zambrini, \bf Progress in Probability\rm ,
vols 58 and 59). More recently it has appeared that each \it one--factor affine interest rate model \rm (in the sense of Leblanc-Scaillet) could be described using such a Bernstein process.

MSC 91G30, 60H10

\end{abstract}
\maketitle

\section{Introduction}
\label{sec:1}

The relationship between Bernstein processes and Mathematical Finance was first glimpsed in
\S 5 of \cite{9}, where the Alili--Patie family of transformations $S^{\alpha,\beta}$
(see \cite{1},\S 2, for the particular case $\alpha=1$, and \cite{9}, pp.60--61, for the general case) was reinterpreted in the context of isovectors for the (backward) heat equation. Furthermore, Patie's composition formula
(\cite{9}, p.60) for these transformations was derived from the commutation relations of the canonical basis
of the aforementioned isovector algebra.

As the motivation for Patie's work had been to compute certain option prices in the framework of an affine interest rate model
(see \cite{9}, pp. 101--104), it was natural to look for a \it direct parametrization \rm of such a model by a Bernstein process.
It turns out that, for any \it one--factor interest rate model \rm in the sense of Leblanc and Scaillet (\cite{4},p. 351), the associated \it square root process \rm coincides, up to some (possibly infinite) random time, with a Bernstein process
(\S 3, Theorem 3.4). The potential $\displaystyle\frac{1}{q^{2}}$
(that had been considered in \cite{8}, p. 220, on physical grounds)
appears here naturally (Theorem 3.4) : this is one more example of the 
deep links between Euclidean Quantum Mechanics and Mathematical Finance.

\section{An isovector calculation}
\label{sec:2}

We shall work within the context of \cite{8}, \S 3.
We consider the Hamilton--Jacobi--Bellman equation 
\begin{equation}
\displaystyle\frac{\partial S}{\partial t}+\displaystyle\frac{\theta^{2}}{2}\displaystyle\frac{\partial^{2} S}{\partial q^{2}}-\displaystyle\frac{1}{2}
(\displaystyle\frac{\partial S}{\partial q})^{2}+V=0\,\,\, \tag{$\mathcal H\mathcal J\mathcal B^{V}$}
\end{equation}
with potential
$$
V(t,q)=\frac{C}{q^{2}}+Dq^{2}\,\, ,
$$
where we have replaced (as explained in \cite{5}) $\sqrt{\hbar}$ by $\theta$.

In order to determine the Lie algebra $\mathcal H_{V}$ of \it pure isovectors \rm for $(\mathcal H\mathcal J\mathcal B^{V})$, \it i.e. \rm the
algebra $\aleph$ of \cite{8},  we have to solve the auxiliary equation (3.29) from \cite{8},p. 215, \it i.e. \rm :
$$
-\frac{1}{4}\overset{...}{T_{N}}q^{2}-\overset{..}{l}q+\overset{.}{\sigma}+\frac{1}{2}\overset{.}{T_{N}}q(-\frac{C}{q^{3}}+2D q)+l(-\frac{2C}{q^{3}}+2Dq)+\overset{.}{T_{N}}(\frac{C}{q^{2}}+Dq^{2})
-\frac{\theta^{2}}{4}\overset{..}{T_{N}}=0 \,\, ,
$$
that is :
$$
q^{2}(2D\overset{.}{T_{N}}-\frac{1}{4}\overset{...}{T_{N}})
+q(-\overset{..}{l}+2Dl)+\overset{.}{\sigma}-\frac{\theta^{2}}{4}\overset{...}{T_{N}}-\frac{2Cl}{q^{3}}=0\,\, .
$$

As $T_{N}$, $l$ and $\sigma$ depend only upon $t$, the system is equivalent to :
$$
\left\{
\aligned 2Cl &= 0 \\
\overset{..}{l} &= 2Dl \\
\overset{.}{\sigma} &=\frac{\theta^{2}}{4}\overset{..}{T_{N}} \\
\overset{...}{T_{N}} &=8D\overset{.}{T_{N}}\,\, .
\endaligned
\right.
$$

Two different cases now appear :

\bf{1)} $C\neq 0$\rm

Then one must have $l=0$, in which case the second condition
holds automatically, and the system reduces itself to :

$$
\left\{
\aligned
\overset{.}{\sigma} &=\frac{\theta^{2}}{4}\overset{..}{T_{N}} \\
\overset{...}{T_{N}} &=8D\overset{.}{T_{N}}
\endaligned
\right.
$$

\bf{1)a)} $D>0$\rm

Setting $\epsilon=\sqrt{8D}$, we find

$$
\overset{.}{T_{N}}=C_{1}e^{\epsilon t}+C_{2}e^{-\epsilon t}\,\, ,
$$

whence

$$
T_{N}=\frac{C_{1}}{\epsilon}e^{\epsilon t}-\frac{C_{2}}{\epsilon}e^{-\epsilon t}+C_{3}
$$

and

$$
\sigma=\frac{\theta^{2}}{4}\overset{.}{T_{N}}+C_{4}
$$

therefore :
$$
\sigma=\frac{\theta^{2}C_{1}}{4}e^{\epsilon t}+\frac{\theta^{2}C_{2}}{4}e^{-\epsilon t}+C_{4}
$$

where $(C_{j})_{1\leq j \leq 4}$ denote arbitrary (real) constants.
In particular
$$
dim(\mathcal H_{V})=4\,\, .
$$

\bf{1)b)}$D=0$\rm

Then from $\overset{...}{T_{N}}=0$ follows
$$
T_{N}=C_{1}t^{2}+C_{2}t+C_{3}
$$
for constants $C_{1},C_{2},C_{3}$.
Then 
$$
\overset{.}{\sigma}=\frac{\theta^{2}}{4}\overset{..}{T_{N}}=\frac{\theta^{2}C_{1}}{2}
$$
and
$$
\sigma=\frac{\theta^{2}C_{1}}{2}t+C_{4}\,\, .
$$

Therefore, here too, $dim(\mathcal H_{V})=4$; furthermore, 
we get an explicit expression for the isovectors :
$$
N^{t}=T_{N}=C_{1}t^{2}+C_{2}t+C_{3} \,\, ,
$$

\begin{eqnarray}
N^{q} \nonumber 
&=&\frac{1}{2}q\overset{.}{T_{N}}+l \nonumber \\
&=&\frac{1}{2}q(2C_{1}t+C_{2}) \nonumber \\
&=&C_{1}tq+\displaystyle\frac{C_{2}q}{2} \,\, ,\nonumber
\end{eqnarray}

and

\begin{eqnarray}
\aligned
N^{S} \nonumber
&=-\phi \nonumber \\
&=-\frac{1}{4}q^{2}\overset{..}{T_{N}}-q\overset{.}{l}+\sigma \nonumber \\
&=-\frac{1}{4}q^{2}.2C_{1}+\frac{\theta^{2}}{2}C_{1}t+C_{4} \nonumber \\
&=\frac{C_{1}}{2}(\theta^{2}t-q^{2})+C_{4} \,\, .\nonumber 
\endaligned
\end{eqnarray}

A canonical basis for $\mathcal H_{V}$
is thus given by $(M_{i})_{1\leq i\leq 4}$,
where $M_{i}$ is characterized by $C_{j}=\delta_{ij}$ (Kronecker's symbol).
Using the notation of \cite{7}, it appears that
$$
M_{1}=\displaystyle\frac{1}{2}N_{6} \,\, ,
$$

$$
M_{2}=\displaystyle\frac{1}{2}N_{4} \,\, ,
$$

$$
M_{3}=N_{1} \,\, ,
$$

and

$$
M_{4}=-\displaystyle\frac{1}{\theta^{2}}N_{3} \,\, ,
$$
therefore $\mathcal H_{V}$ is generated by $N_{1}$, $N_{3}$, $N_{4}$ and
$N_{6}$. We thereby recover the result of \cite{8}, p. 220, modulo the correction of a misprint.
This list ties in nicely with the symmetry properties of certain diffusions related to Bessel
processes (see \cite{6} for a detailed explanation).

\bf{1)c)}$D<0$\rm

Setting now $\epsilon=\sqrt{-8D}$, we find

$$
\overset{.}{T_{N}}=C_{1}\cos{(\epsilon t)}+C_{2}\sin{(\epsilon t)} \,\, ,
$$

whence

$$
T_{N}=\frac{C_{1}}{\epsilon}\sin{(\epsilon t)}-\frac{C_{2}}{\epsilon}\cos{(\epsilon t)}+C_{3} \,\, ,
$$

and, as above :

$$
\sigma=\frac{\theta^{2}}{4}\overset{.}{T_{N}}+C_{4}\,\, ,
$$

therefore :
$$
\sigma=\frac{\theta^{2}C_{1}}{4}\cos{(\epsilon t)}+\frac{\theta^{2}C_{2}}{4}\sin{(\epsilon t)}+C_{4}
$$

where $(C_{j})_{1\leq j \leq 4}$ denote arbitrary (real) constants.
In particular \,\, ,
$$
dim(\mathcal H_{V})=4 \,\, .
$$

\bf{2)}$C=0$\rm

Then the system becomes

\begin{eqnarray}
\left\{
\aligned
\overset{..}{l} &=2Dl \nonumber \\
\overset{.}{\sigma} &=\frac{\theta^{2}}{4}\overset{..}{T_{N}} \nonumber \\
\overset{...}{T_{N}}&=8D\overset{.}{T_{N}} \,\, .\\
\endaligned
\right.
\end{eqnarray}

The equation for $l$ on the one hand, and the system for $(\sigma,T_{N})$ on the other
hand, are independent, and, as above, the first one has a two--dimensional 
space of solutions and the second one a four--dimensional space of solutions, \it i.e. \rm

$$
dim(\mathcal H_{V})=6\,\, .
$$

Whence
\begin{theorem} The isovector algebra $\mathcal H_{V}$ associated with $V$
has dimension $6$ if and only if $C=0$ ; in the opposite case,
it has dimension $4$.
\end{theorem}

\newpage

\section{Parametrization of a one--factor affine model}
\label{sec:3}

As general references we shall use, concerning Bernstein processes, our recent survey (\cite{5}),
and, concerning affine models, H\'enon's PhD thesis(\cite{3}) as well as Leblanc and Scaillet's seminal paper (\cite{4}).

An \it one--factor affine interest rate model \rm  is characterized by the instantaneous rate $r(t)$, satisfying
the following stochastic differential equation :
\begin{eqnarray}        
dr(t)=\sqrt{\alpha r(t) + \beta} \,\, dw(t) +(\phi - \lambda r(t)) \,\, dt 
\end{eqnarray}
under the risk--neutral probability $Q$
($\alpha=0$ corresponds to the so--called Vasicek model, and $\beta=0$ corresponds to the
Cox--Ingersoll--Ross model ; cf. \cite{4}).

Assuming $\alpha>0$, let us set $$\tilde{\phi}=_{def}\phi + \displaystyle\frac{\lambda \beta}{\alpha}\,\, ,$$  $$\delta=_{def}\displaystyle\frac{4\tilde{\phi}}{\alpha}\,\, ,$$ and let us also assume
that $\tilde{\phi}\geq 0$.

The following two quantities will play an important role :

\begin{eqnarray}
C 
&:=&\frac{\alpha^{2}}{8}(\tilde{\phi}-\frac{\alpha}{4})(\tilde{\phi}-\frac{3\alpha}{4}) \nonumber \\
&=&\frac{\alpha^{4}}{128}(\delta-1)(\delta-3) \nonumber
\end{eqnarray}

and

$$
D:=\frac{\lambda^{2}}{8}\,\, .
$$

Let us set $X_{t}=\alpha r(t)+\beta$.
\begin{proposition}Let $r_{0}\in \mathbf R$ ;
then the stochastic differential equation
\begin{eqnarray}
dr(t)=\sqrt{\vert \alpha r(t) + \beta \vert} \,\, dw(t) +(\phi - \lambda r(t)) \,\, dt 
\end{eqnarray}
has a unique strong solution such that $r(0)=r_{0}$.
Furthermore, in case that $$\alpha r_{0}+\beta \geq 0\,\, ,$$ 
 one has  $\alpha r(t)+\beta \geq 0$ for all $t\geq 0$ ;
in particular, $r(t)$ satisfies $(3.1)$.
\end{proposition}
\begin{proof}

Let us set $X_{t}=\alpha r(t)+\beta$ ; it is easy to see that,
in terms of $X_{t}$, equation $3.2$ becomes:

\begin{eqnarray}
dX_{t} 
&=&\alpha dr(t) \nonumber \\
&=&\alpha(\sqrt{\vert X_{t} \vert}dw(t)+(\phi-\lambda\frac{X_{t}-\beta}{\alpha})dt) \nonumber \\
&=&\alpha\sqrt{\vert X_{t} \vert}dw(t)+(\alpha\tilde{\phi}-\lambda X_{t})dt  \,\, . \nonumber
\end{eqnarray}

We are therefore in the situation of $(1)$, p.313, in \cite{2}, with $c=\alpha$,
$a=\alpha\tilde{\phi}$, and $b=-\lambda$ ; the result follows.

In case $\lambda\neq 0$, one may also refer to \cite{3},p.55, Proposition 12.1, with $\sigma=\alpha$, $\kappa=\lambda$ and
$a=\displaystyle\frac{\alpha\tilde{\phi}}{\lambda}$.
\end{proof}

We shall henceforth assume all of the hypotheses of Proposition 3.1 to be satisfied.
\begin{corollary}One has
$$
\left\{
\begin{array}{lr}
X_{t}=e^{-\lambda t} \, Y(\frac{\alpha^{2}(e^{\lambda t}-1)}{4\lambda}) \,\, \text{for} \,\, \lambda \neq 0 \nonumber \,\, , \text{and} \\
X_{t}=Y(\frac{\alpha^{2}t}{4}) \,\, \text{for} \,\, \lambda = 0 \nonumber \\
\end{array}
\right.
$$
where $Y$ is a $BESQ^{\delta}$(squared Bessel process with parameter 
$\delta$) having initial value $Y_{0}=\alpha r_{0}+\beta$.
\end{corollary}
\begin{proof}In case $\lambda\neq 0$, one applies the result of \cite{3}, p. 314.
For $\lambda=0$, let
$$
Z_{t}:=\displaystyle\frac{4}{\alpha^{2}}X_{t}\,\, ;
$$
it appears that 
$$
dZ_{t}=2\sqrt{\vert Z_{t} \vert}dw(t)+\delta dt \,\, ,
$$
whence $Z_{t}$ is a $BESQ^{\delta}$--process.
As
$$
X_{t}:=\displaystyle\frac{\alpha^{2}}{4}Z_{t}\,\, ,
$$
the scaling property of Bessel processes 
yields the result.
\end{proof}

\begin{theorem}If $\delta\geq 2$ one has, almost surely :
$$
\forall t>0  \,\,\, X_{t}>0\,\, ;
$$
on the other hand, if $\delta<2$, almost surely there is a $t>0$ such that $X_{t}=0$.
\end{theorem}
\begin{proof}We apply Corollary 1, p. 317, from \cite{4}(12.2)  yielding that 
\begin{eqnarray}
\forall t>0 \,\,\, X_{t}>0 \,\, .
\end{eqnarray}

One may also use \cite{3}, p.56, from which follows that $(3.3)$
is equivalent to 
$$\displaystyle\frac{2a\kappa}{\sigma^{2}}\geq 1 \,\, ; $$
 but,
according to the above identifications,
$$
\displaystyle\frac{2a\kappa}{\sigma^{2}}=\displaystyle\frac{2{\displaystyle\frac{\alpha\tilde{\phi}}{\lambda}}{\lambda}}{\alpha^{2}}
=\displaystyle\frac{2\tilde{\phi}}{\alpha}=\displaystyle\frac{\delta}{2}\,\, .
$$
\end{proof}

Our main result is the following :
\begin{theorem} Let us define the process
$$
z(t)=\sqrt{X_{t}}
$$
and the stopping time
$$
T=\inf\{t>0 \vert X_{t}=0\} ;
$$
as seen in Theorem 3.3, $T=+\infty \,\, 
a.s.$ for $\delta\geq 2$,
and 
$T<+\infty \,\, a.s.$ for $\delta<2$.
\, 
Then there exists a Bernstein process $y(t)$
for 
$$
\theta=\frac{\alpha}{2}\,\,
$$
and the potential
$$
V(t,q)=\frac{C}{q^{2}}+Dq^{2}\,\, .
$$
such that
$$
\forall t\in [0,T[\,\, z(t)=y(t) \,\, .
$$
In particular, for $\delta\geq 2$, $z$ itself is a Bernstein process.
\end{theorem}

\begin{proof}
One has (cf. Proposition 3.1 and its proof)

\begin{eqnarray}
dX_{t} \nonumber
&=&\alpha\sqrt{X_{t}}dw(t)+(\alpha\tilde{\phi}-\lambda X_{t})dt \nonumber \\
&=&\alpha z(t)dw(t)+(\alpha\tilde{\phi}-\lambda z(t)^{2})dt \,\, .\nonumber
\end{eqnarray}

Taking now  $f(x)=\sqrt{x}$, we have $$\forall x>0\,\, f^{'}(x)=\frac{1}{2\sqrt{x}}\,\,\text{and}
\,\,f^{''}(x)=-\frac{1}{4}x^{-\frac{3}{2}}\,\, ,$$ therefore, for all $t\in ]0,T[$,
$f^{'}(X_{t})=\displaystyle\frac{1}{2z(t)}$
and $f^{''}(X_{t})=-\frac{1}{4}z(t)^{-3}$.
The application of It\^o's formula now gives :

\begin{eqnarray}
dz(t) \nonumber
&=&d(f(X_{t})) \nonumber \\
&=&f^{'}(X_{t})dX_{t}+\frac{1}{2}f^{''}(X_{t})(dX_{t})^{2} \nonumber \\
&=&\frac{1}{2z(t)}(\alpha z(t)dw(t)+(\alpha\tilde{\phi}-\lambda z(t)^{2})dt)
-\frac{1}{8}z(t)^{-3}\alpha^{2}z(t)^{2}dt \nonumber \\
&=&\displaystyle\frac{\alpha}{2}dw(t)+\frac{1}{8z(t)}(4\alpha\tilde{\phi}-4\lambda z(t)^{2}-\alpha^{2})dt \,\, .\nonumber 
\end{eqnarray}

Let us now define $\eta$ by
\begin{eqnarray}
\eta(t,q) 
&:=&e^{\displaystyle\frac{\lambda\tilde{\phi}t}{\alpha}-\displaystyle\frac{\lambda q^{2}}{\alpha^{2}}}
q^{\displaystyle\frac{2\tilde{\phi}}{\alpha}-\displaystyle\frac{1}{2}}
\,\, \nonumber \\
&=&e^{\displaystyle\frac{\lambda \delta t}{4}-\displaystyle\frac{\lambda q^{2}}{\alpha^{2}}}
q^{\displaystyle\frac{\delta - 1}{2}} \nonumber \,\, .
\end{eqnarray}

It is easy to check that $\eta$ solves the equation
$$
\theta^{2}\displaystyle\frac{\partial \eta}{\partial t}=-\displaystyle\frac{\theta^{4}}{2}
\displaystyle\frac{\partial^{2}\eta}{\partial q^{2}}+V\eta  
$$
for
$$
V=\frac{C}{q^{2}}+Dq^{2}\,\, ;
$$
in other words,
\begin{eqnarray}
S \nonumber
&:=&-\theta^{2}\ln(\eta) \nonumber \\
&=&-\theta^{2}(\displaystyle\frac{\lambda \delta t}{4}-\displaystyle\frac{\lambda q^{2}}{\alpha^{2}}
+\displaystyle\frac{\delta - 1}{2}\ln(q)) \nonumber \,\, \\
&=&-\displaystyle\frac{\alpha^{2}\lambda \delta t}{16}+\displaystyle\frac{\lambda q^{2}}{4}
-\alpha^{2}(\frac{\delta - 1}{8})\ln(q) \nonumber \,\, 
\end{eqnarray}
satisfies $(\mathcal H\mathcal J\mathcal B^{V})$.
Furthermore we 
have :

\begin{eqnarray}
\tilde{B} 
&:=&\theta^{2}\displaystyle\frac{\displaystyle\frac{\partial \eta}{\partial q}}{\eta} \nonumber \\
&=&-\displaystyle\frac{\partial S}{\partial q} \nonumber \\ 
&=&-\displaystyle\frac{\lambda q}{2}+\displaystyle\frac{\alpha^{2}(\delta - 1)}{8q} \nonumber \\
&=& \displaystyle\frac{1}{8q}(\alpha^{2}\delta-\alpha^{2}-4\lambda q^{2}) \nonumber
\end{eqnarray}
whence
$$
\tilde{B}(t,z(t))=\displaystyle\frac{1}{8z(t)}(4\alpha\tilde{\phi}-\alpha^{2}-4\lambda z(t)^{2})
$$
and $z$ satisfies the stochastic differential equation associated with $\eta$ :
$$
\forall t\in ]0,T[ \,\,\, dz(t)=\theta dw(t)+\tilde{B}(t,z(t))dt
$$
(as in \cite{5}, \S 1, equation ($\mathcal B)$) ; the result follows.
\end{proof}

\begin{proposition} The isovector algebra $\mathcal H_{V}$ associated with $V$
has dimension $6$ if and only if $\tilde{\phi}\in\{\frac{\alpha}{4},\frac{3\alpha}{4}\}$, i.e. $\delta\in \{1,3\}$ ; in the opposite case,
it has dimension $4$.
\end{proposition}
\begin{proof}
It is enough to apply Theorem 2.1, observing that the condition $C=0$ is equivalent to $\tilde{\phi}\in \{\frac{\alpha}{4},\frac{3\alpha}{4}\}$.
\end{proof}

In the context of H\'enon's already mentioned PhD thesis (\cite{3},p.55)
we have \newline $\phi=\kappa a$, $\lambda=\kappa$, $\alpha=\sigma^{2}$ et $\beta=0$,
whence $\tilde{\phi}=\kappa a$ and the condition $C=0$
is equivalent to
$$\kappa a\in\{\frac{\sigma^{2}}{4},\frac{3\sigma^{2}}{4}\}\,\, .$$
Let us analyze more closely the situation in which $C=0$ ;
the general case will be commented upon in \cite{6}.

\bf{1)}$\tilde{\phi}=\displaystyle\frac{\alpha}{4}$\rm, \it i.e. \rm $\delta=1\, .$

Then $y(t)$ is a solution of
$$
dy(t)=\displaystyle\frac{\alpha}{2}dw(t)-\frac{\lambda}{2}y(t)dt \nonumber \,\, ,
$$
\it i.e. \rm $y(t)$ is an Ornstein--Uhlenbeck process
(it was already known that the Ornstein--Uhlenbeck process was a Bernstein process for a quadratic potential).
Therefore $z(t)$ co\"\i ncides, on the random interval $[0,T[$, with an Ornstein--Uhlenbeck
process.
Here 
$$
\eta(t,q)=e^{\displaystyle\frac{\lambda t}{4}-\displaystyle\frac{\lambda q^{2}}{\alpha^{2}}}\,\, .
$$
From
\begin{eqnarray}
y(t)\nonumber
&=&e^{-\frac{\lambda t}{2}}(y_{0}+\frac{\alpha}{2}\int_{0}^{t}e^{\frac{\lambda s}{2}}dw(s)) \nonumber \\
&=&e^{-\frac{\lambda t}{2}}(z_{0}+\tilde{w}(\frac{\alpha^{2}(e^{\lambda t}-1)}{4\lambda})) \nonumber
\end{eqnarray}
($\tilde{w}$ denoting another Brownian motion),
it appears that $y(t)$ follows a normal law with mean $e^{-\frac{\lambda t}{2}}z_{0}$ and variance $\frac{\alpha^{2}(1-e^{-\lambda t})}{4\lambda}$.
The density $\rho_{t}(q)$ of $y(t)$ is therefore given by :
$$
\rho_{t}(q)=\displaystyle\frac{2\sqrt{\lambda}}{\alpha\sqrt{2\pi(1-e^{-\lambda t})}}
\exp{(-\displaystyle\frac{2\lambda(q-e^{-\frac{\lambda t}{2}}z_{0})^{2}}{\alpha^{2}(1-e^{-\lambda t})})}\,\, .
$$
Whence
\begin{eqnarray}
\forall t>0 \,\,\,\,
\eta_{*}(t,q) 
&=&\displaystyle\frac{\rho_{t}(q)}{\eta(t,q)} \nonumber \\
&=&\displaystyle\frac{1}{\alpha}\sqrt{\displaystyle\frac{\lambda}{\pi\sinh{(\displaystyle\frac{\lambda t}{2})}}}e^{(\displaystyle\frac{-\lambda q^{2}-\lambda q^{2}e^{-\lambda t}+4\lambda qz_{0}e^{-\frac{\lambda t}{2}}-2\lambda z_{0}^{2}e^{-\lambda t}}{\alpha^{2}(1-e^{-\lambda t})})} \nonumber
\end{eqnarray} 
and one may check that, as was to be expected, $\eta_{*}$ satisfies the following equation
($\mathcal C_{2}^{(V)}$ in \cite{5}) :
 
\begin{eqnarray}
-\theta^{2}\displaystyle\frac{\partial \eta_{*}}{\partial t}=-\displaystyle\frac{\theta^{4}}{2}
\displaystyle\frac{\partial^{2}\eta_{*}}{\partial q^{2}}+V\eta_{*}          \,\, .
\end{eqnarray}

\bf{2)}$\tilde{\phi}=\displaystyle\frac{3\alpha}{4}$\rm, \it i.e. \rm $\delta=3$.

In that case, according to Theorem 3.2, $T=+\infty$ whence $y=z$.
Furthermore
$$
\eta(t,q)=q e^{\displaystyle\frac{\lambda}{\alpha^{2}}(\displaystyle\frac{3\alpha^{2}t}{4}-q^{2})}\,\, .
$$
Let us define
$$
s(t)=e^{-\displaystyle\frac{\lambda t}{2}}\frac{1}{z(t)}\,\, ;
$$
then an easy computation, using It\^o's formula in the same way as above, shows that
$$
ds(t)=-\frac{\alpha}{2}e^{\frac{\lambda t}{2}}s(t)^{2}dw(t)\,\, ;
$$
in particular, $s(t)$ is a martingale.

Referring once more to Proposition 3.1 and its proof, we see that
$$
dX_{t}=\alpha \sqrt{X_{t}}dw(t)+(\frac{3\alpha^{2}}{4}-\lambda X_{t})dt\,\, .
$$

Let us now assume $X_{0}=0$ and $\lambda \neq 0$ ; then, according to Corollary 3.2,
$$
X_{t}=e^{-\lambda t} Y(\frac{\alpha^{2}(e^{\lambda t}-1)}{4\lambda})
$$
where $Y$ is a $BESQ^{3}$(squared Bessel process with parameter 
$3$) such that $Y(0)=0$.
But, for each fixed $t>0$, $Y_{t}$ has the same law as $tY_{1}$, and $Y_{1}=\vert\vert B_{1}\vert\vert^{2}$
is the square of the norm of a $3$--dimensional Brownian motion ; the law of $Y_{1}$
is therefore
$$
\frac{1}{\sqrt{2\pi}}e^{-\frac{u}{2}}\sqrt{u}\mathbf 1_{u\geq 0}du\,\, .
$$
Therefore the density $\rho_{t}(q)$ of the law of $z(t)$ is given by :
$$
\rho_{t}(q)=\frac{1}{\sqrt{2\pi}}\frac{16\lambda^{\frac{3}{2}}}{\alpha^{3}(1-e^{-\lambda t})^{\frac{3}{2}}}q^{2}e^{-\displaystyle\frac{2\lambda q^{2}}{\alpha^{2}(1-e^{-\lambda t})}}
$$
and
$$
\forall t>0 \,\, \eta_{*}(t,q)=\frac{\rho_{t}(q)}{\eta(t,q)}=
\frac{16\lambda^{\frac{3}{2}}}{\alpha^{3}\sqrt{2\pi}}(1-e^{-\lambda t})^{-\frac{3}{2}}qe^{-\displaystyle\frac{3\lambda t}{4}-\displaystyle\frac{\lambda q^{2}}{\alpha^{2}\tanh(\frac{\lambda t}{2})}} \,\, .
$$

Here, too, one may check directly that $\eta_{*}$ satisfies equation 
$(3.4)$ above.
\newpage
\section{Acknowledgements}

I am grateful to the colleagues who invited me to present preliminary versions
of this work and thereby provided me with a much--needed moral support : Professor Barbara R\"udiger (Koblenz, July 2007), Professor Pierre Patie (Bern, January 2009) and Professors Paul Bourgade and Ali S\"uleiman Ust\"unel
(Institut Henri Poincar\'e, February 2009). Comments by Professor Pierre Patie led to many improvements in the
formulations. I am also indebted to Mohamad Houda for a careful reading of previous versions
of the paper.

\end{document}